\documentclass[11pt,a4paper]{article}

\usepackage[margin=2.6cm]{geometry}
\usepackage{amsmath,amssymb,amsthm,mathtools}
\usepackage{array,booktabs}
\usepackage{enumitem}
\IfFileExists{newtxtext.sty}{%
	\IfFileExists{newtxmath.sty}{\usepackage{newtxtext,newtxmath}}{}%
}{}
\usepackage{microtype}
\usepackage[hidelinks]{hyperref}
\newcommand{\R}{\mathbb{R}}

\newcommand{\Sph}{\mathbb{S}}
\DeclareMathOperator{\Rm}{Rm}

\DeclareMathOperator{\tr}{tr}
\DeclareMathOperator{\Div}{div}
\newcommand{\dd}{\,\mathrm d}
\newcommand{\norm}[1]{\left\lVert #1\right\rVert}

\numberwithin{equation}{section}
\theoremstyle{plain}
\newtheorem{theorem}{Theorem}[section]
\newtheorem{lemma}[theorem]{Lemma}
\newtheorem{proposition}[theorem]{Proposition}
\theoremstyle{definition}
\newtheorem{example}[theorem]{Example}

\newenvironment{pf}[1][Proof]{%
	\par\addvspace{0.5\baselineskip}%
	\noindent\textit{#1}\quad\ignorespaces
}{%
	\unskip\nobreak\hfill$\square$\par\addvspace{0.3\baselineskip}%
}

\newenvironment{remark}[1][]{%
	\par\addvspace{0.5\baselineskip}%
	\noindent\textbf{Remark\if\relax\detokenize{#1}\relax\else\space(#1)\fi}\quad\ignorespaces
}{%
	\par\addvspace{0.3\baselineskip}%
}

\hypersetup{
	pdftitle={Total Mean Curvature Gaps for Scalar-Flat Hypersurfaces},
	pdfauthor={Jiuru Zhou},
	hidelinks
}

\begin{document}
	
	%==================== Title / author / dedication ====================
	\begin{center}
		{\Large\bfseries Total Mean Curvature Gaps for Scalar-Flat Hypersurfaces}\par
		
		\vspace{1.0em}
	\end{center}
	\vspace{0.8em}
	
	%==================== Affiliation ====================
	\begin{center}
		Jiuru Zhou\\[0.2em]
		{\itshape Dedicated to Professor Xingwang Xu on his 70th birthday.}
	\end{center}
	
	%==================== Abstract / keywords / classification ====================
	\begin{center}
		\begin{minipage}{0.94\textwidth}
			\small
			\noindent\textbf{Abstract}\quad We prove that a complete connected
			scalar-flat Euclidean hypersurface of dimension $n\geq3$ is an
			affine hyperplane if its normalized mean curvature has sufficiently
			small $L^n$ norm. We also
			compute the total mean curvature energy of a classical
			rotational scalar-flat family to show that finiteness cannot
			replace smallness in dimensions $n\geq3$.
			
			
		\end{minipage}
	\end{center}
	\vspace{1.4em}
	\section{Introduction}\label{sec:introduction}
	
	Integral curvature gap theorems relate the geometry of a complete
	submanifold to the size of its curvature. In Euclidean space, the
	curvature identities of Simons~\cite{Simons} and the Sobolev
	inequalities of Michael--Simon~\cite{MichaelSimon} and
	Hoffman--Spruck~\cite{HoffmanSpruck} are basic tools in this study.
	Critical integral estimates for the second fundamental form give
	rigidity for minimal submanifolds and for submanifolds with parallel
	mean curvature; see Ni~\cite{Ni} and Xu--Gu~\cite{XuGu}. Small mean
	curvature alone does not imply flatness, since a complete minimal
	submanifold need not be flat.
	
	Throughout, manifolds are smooth, connected, and without boundary,
	and the metric is induced by the immersion. For
	$F:(M^n,g)\to\R^{n+m}$, let $\mathrm{II}$ be the normal-valued
	second fundamental form and let $H=n^{-1}\tr_g\mathrm{II}$ be the
	normalized mean curvature vector. We use the scalar curvature
	$R=\sum_{i,j}\Rm_{ijij}$ and the full tensor norm of $\Rm$.
	The contracted Gauss equation gives
	\begin{equation}\label{eq:contracted-gauss}
		R=n^2|H|^2-|\mathrm{II}|^2.
	\end{equation}
	Hence
	\begin{equation}\label{eq:critical-equivalence}
		R=0\quad\Longrightarrow\quad
		|\mathrm{II}|=n|H|,\qquad
		\norm{\mathrm{II}}_{L^n(M,g)}=n\norm{H}_{L^n(M,g)}.
	\end{equation}
	Both norms are invariant under ambient dilations. We refer to
	$\norm{H}_{L^n(M,g)}$ as the total mean curvature norm; the corresponding
	integral $\int_M|H|^n\dd v_g$ is its critical energy. Under scalar flatness,
	smallness of this norm is therefore equivalent to smallness of the
	$L^n$ norm of the second fundamental form.
	
	For hypersurfaces, the Newton transformations provide a classical
	approach to constant scalar curvature; see Reilly~\cite{Reilly} and
	Cheng--Yau~\cite{ChengYau}. Under scalar flatness and small critical
	mean-curvature norm, Li--Xu--Zhou~\cite[Theorem~1.2]{LiXuZhou} proved
	the equivalence of local conformal flatness, intrinsic flatness, and
	two equality conditions for the shape tensor. That result does not
	assert that these equivalent conditions hold without an additional
	hypothesis. Their Theorem~1.1 gives a hyperplane theorem under Ricci
	flatness and finite total mean curvature norm.
	Bai--Luo~\cite{BaiLuo} obtained a hyperplane gap under the additional
	differential inequality
	\[
	-P_1(\nabla H,\nabla|H|)
	\leq-\delta|H||\nabla H|^2,\qquad P_1=nHg-A,
	\]
	for some $\delta>0$. Related rigidity results under stability
	assumptions were obtained by Silva Neto~\cite{SilvaNeto} and
	Alencar--do Carmo--Silva Neto~\cite{AlencarDoCarmoNeto}.
	
	Our main result removes the local conformal flatness and differential
	assumptions from these scalar-flat gap results. Its new assertion is
	the gap in dimensions $n\geq3$; the two-dimensional statement is a
	consequence of the classical cylinder theorem.
	
	\begin{theorem}[The hypersurface gap]
		\label{h:thm:main}
		Let $n\geq2$ and let $F:(M^n,g)\to\R^{n+1}$ be a complete
		connected immersion with $R\equiv0$. Assume that
		\begin{enumerate}[label=\textup{(\roman*)},nosep]
			\item $\int_M|H|^q\dd v_g<\infty$ for some $q>0$ if $n=2$;
			\item $\norm{H}_{L^n(M,g)}<\beta_n$ if $n\geq3$, where $\beta_n>0$
			depends only on $n$ and is specified in~\eqref{h:eq:beta}.
		\end{enumerate}
		Then $M$ is isometric to $\R^n$ and $F$ is an isometric embedding
		onto an affine hyperplane. More precisely, there are an isometry
		$\Psi:\R^n\to M$ and a Euclidean rigid motion $Q_0$ of $\R^{n+1}$
		such that
		\begin{equation}\label{h:eq:hyperplaneclassification}
			Q_0\circ F\circ\Psi(x)=(x,0),\qquad x\in\R^n.
		\end{equation}
	\end{theorem}
	
	\begin{remark}
		For $n\geq3$, fix any admissible dimensional constant $C_n$
		in the Michael--Simon inequality~\eqref{h:eq:ms}. An explicit
		threshold in terms of this choice is
		\begin{equation}\label{h:eq:beta}
			\beta_n=\frac{n-2}{8n(n-1)C_n\sqrt{2n(n^2+8)}}.
		\end{equation}
		This constant is not claimed to be sharp.
	\end{remark}
	
	\noindent\emph{Strategy of the proof.} Write
	\[
	D=|\nabla A|^2-n^2|\nabla H|^2,\qquad
	E=n^4H^4-nH\tr(A^3).
	\]
	Scalar flatness gives $|A|=n|H|$, $D\geq0$, and comparability of
	$E$ with $|\Rm|^2$ (Lemma~\ref{h:lem:algebra}). The scalar-flat
	Simons identity, recalled from~\cite[Lemma~2.2]{LiXuZhou}, becomes
	\[
	\Div\bigl(n(nHg-A)\nabla H\bigr)=D-E.
	\]
	The key additional estimate in Lemmas~\ref{h:lem:codazzi}
	and~\ref{h:prop:weightedcurvature} controls
	$|\nabla(|A|^{(n-4)/2}\Rm)|^2$ by $|A|^{n-2}D$.
	Integrating the divergence identity against $\phi^2|A|^{n-2}$
	yields a weighted energy inequality. The Michael--Simon Sobolev
	inequality, in the form used in~\cite[Lemma~2.5]{LiXuZhou}, and
	H\"older's inequality then permit absorption of the curvature term
	when $\norm{H}_{L^n}$ is small. Intrinsic cutoffs give $\Rm=0$.
	Finally, the Hartman--Nirenberg cylinder theorem~\cite{HartmanNirenberg}
	and integrability of $H$ exclude all bending and identify the
	immersion globally with an affine hyperplane.
	
	In the terminology of~\cite[Theorem~1.2]{LiXuZhou}, all the equivalent
	conditions there are thus satisfied under scalar flatness and a
	sufficiently small total mean curvature norm. No further
	differential assumption as in~\cite{BaiLuo} is needed. The classical
	rotational family in Example~\ref{prop:finite-test} has finite
	positive critical energy, so finiteness cannot replace smallness
	for $n\geq3$.
	
	Section~\ref{sec:preliminaries} establishes the pointwise estimates.
	Sections~\ref{sec:gap} and~\ref{sec:classification} prove
	Theorem~\ref{h:thm:main}. Section~\ref{sec:scalar-necessity} tests its
	hypotheses, and Section~\ref{sec:questions} records two further questions.
	
	\section{Sobolev control and scalar-flat hypersurface identities}
	\label{sec:preliminaries}
	
	This section and Section~\ref{sec:gap} assume
	$n\geq3$. For a hypersurface, choose a local unit normal $\nu$ and identify its
	second fundamental form with the symmetric tensor $A$ given by
	$\mathrm{II}=A\nu$. In the hypersurface calculations, the same symbol
	$H$ denotes the normalized scalar mean curvature:
	\begin{equation}\label{h:eq:trace}
		H=\frac1n\tr_g A.
	\end{equation}
	The Gauss equation gives
	\begin{equation}\label{h:eq:gauss}
		\Rm_{ijkl}=A_{ik}A_{jl}-A_{il}A_{jk},
		\qquad R=n^2H^2-|A|^2.
	\end{equation}
	
	\begin{remark}[Orientation]
		If $M$ is not orientable, $A$ and $H$ are defined only up to the simultaneous
		change $(A,H)\mapsto(-A,-H)$ coming from the choice of the local unit normal.
		Every quantity used below is invariant under this change: $|A|$, $|H|$,
		$nH(nHg-A)$, $nH\tr(A^3)$, $|\nabla A|^2-n^2|\nabla H|^2$, the form
		$(nHg-A)(\nabla H,\cdot)$, and hence the vector field $n(nHg-A)\nabla H$.
		Thus all identities below are identities between globally defined objects on
		$M$, and the integrations by parts are legitimate.
	\end{remark}
	
	\subsection{The critical Sobolev inequality}
	
	Fix a dimensional constant $C_n>0$ in the Michael--Simon
	inequality~\cite{MichaelSimon}; see also~\cite{HoffmanSpruck}:
	\begin{equation}\label{h:eq:ms}
		\left(\int_M\eta^{n/(n-1)}\right)^{(n-1)/n}
		\leq C_n\left(\int_M|\nabla\eta|+n\int_M|H|\eta\right)
	\end{equation}
	for every nonnegative $\eta\in C_0^1(M)$.
	
	If
	\begin{equation}\label{h:eq:sobolevcondition}
		nC_n\norm{H}_{L^n(M)}<1,
	\end{equation}
	then, as shown below (cf.~Lemma~2.5 of Li, Xu, and Zhou~\cite{LiXuZhou}), we have
	\begin{equation}\label{h:eq:sobolev}
		\left(\int_M|\varphi|^{2n/(n-2)}\right)^{(n-2)/n}
		\leq C_S\int_M|\nabla\varphi|^2,
		\qquad \varphi\in C_0^1(M),
	\end{equation}
	where
	\[
	C_S=
	\left(\frac{2(n-1)C_n}
	{(n-2)(1-nC_n\norm{H}_{L^n(M)})}\right)^2.
	\]
	In particular,
	\begin{equation}\label{h:eq:sobolevbound}
		\norm{H}_{L^n(M)}<\frac1{2nC_n}
		\quad\Longrightarrow\quad
		C_S\leq\frac{16(n-1)^2C_n^2}{(n-2)^2}.
	\end{equation}

	\subsection{Algebraic consequences of scalar flatness}
	
	If $R=0$, then
	\begin{equation}\label{h:eq:scalarflat}
		|A|^2=n^2H^2.
	\end{equation}
	In particular, $H=0$ if and only if $A=0$, and the Gauss equation
	implies $\Rm=0$ at every such point.
	
	\begin{lemma}\label{h:lem:algebra}
		On a scalar-flat Euclidean hypersurface,
		\begin{equation}\label{h:eq:newtonbounds}
			0\leq nH(nHg-A)
			\leq\frac{2(n-1)}n|A|^2g.
		\end{equation} Moreover,
		\begin{align}
			\frac4n\bigl(n^4H^4-nH\tr(A^3)\bigr)
			&\leq|\Rm|^2
			\leq4\bigl(n^4H^4-nH\tr(A^3)\bigr),\label{h:eq:algebra}\\
			|\Rm|^2&\leq2|A|^4.\label{h:eq:gaussnorm}
		\end{align}
	\end{lemma}
	
	\begin{pf}
		At a point where $H\neq0$, diagonalize $A$ and denote its principal
		curvatures by $\lambda_1,\ldots,\lambda_n$. The trace identity and
		\eqref{h:eq:scalarflat} give
		\[
		\sum_{i=1}^n\frac{\lambda_i}{nH}=1,\qquad
		\sum_{i=1}^n\left(\frac{\lambda_i}{nH}\right)^2=1.
		\]
		For each $i$, the Cauchy--Schwarz inequality applied to the other
		$n-1$ terms yields
		\[
		\left(\frac{\lambda_i}{nH}\right)^2+
		\frac{1}{n-1}\left(1-\frac{\lambda_i}{nH}\right)^2\leq1.
		\]
		Solving this gives
		\begin{equation}\label{h:eq:principalrange}
			-\frac{n-2}{n}\leq\frac{\lambda_i}{nH}\leq1.
		\end{equation}
		The eigenvalues of $nH(nHg-A)$ are
		$|A|^2(1-\lambda_i/(nH))$. Thus~\eqref{h:eq:principalrange} proves
		\eqref{h:eq:newtonbounds}.
		
		The same trace and norm identities give
		\[
		n^4H^4-nH\tr(A^3)
		=n^4H^4\sum_{i=1}^n
		\left(\frac{\lambda_i}{nH}\right)^2
		\left(1-\frac{\lambda_i}{nH}\right)\geq0.
		\]
		In addition,
		\begin{align*}
			|\Rm|^2
			&=2\bigl(|A|^4-\tr(A^4)\bigr)\\
			&=2n^4H^4\sum_{i=1}^n
			\left(\frac{\lambda_i}{nH}\right)^2
			\left(1-\frac{\lambda_i}{nH}\right)
			\left(1+\frac{\lambda_i}{nH}\right).
		\end{align*}
		By~\eqref{h:eq:principalrange}, the last factor lies between $2/n$
		and $2$. This proves~\eqref{h:eq:algebra}. The first line also gives
		\eqref{h:eq:gaussnorm}. At $H=0$, both $A$ and $\Rm$ vanish, so all
		the inequalities hold there also.
	\end{pf}
	
	\subsection{Codazzi estimates for the curvature tensor}
	
	\begin{lemma}[The Codazzi defect]\label{h:lem:codazzi}
		On the set where $H\neq0$,
		\begin{align}
			|\nabla A|^2-n^2|\nabla H|^2
			&=n^2H^2\left|\nabla\left(\frac{A}{nH}\right)\right|^2\geq0,
			\label{h:eq:energydefect}\\
			n^2|\nabla H|^2\left|\frac{\Rm}{n^2H^2}\right|^2
			&\leq8\bigl(|\nabla A|^2-n^2|\nabla H|^2\bigr),
			\label{h:eq:codazziestimates}\\
			\left|\nabla\left(\frac{\Rm}{n^2H^2}\right)\right|
			&\leq4\left|\nabla\left(\frac{A}{nH}\right)\right|.
			\label{h:eq:normalizedgauss}
		\end{align}
		The difference $|\nabla A|^2-n^2|\nabla H|^2$ is zero wherever
		$H=0$.
	\end{lemma}
	
	\begin{pf}
		Differentiating~\eqref{h:eq:scalarflat} in the direction of a local
		orthonormal frame vector gives
		\[
		\sum_{i,j}A_{ij}\nabla_kA_{ij}=n^2H\nabla_kH.
		\]
		Consequently,
		\begin{align*}
			n^2H^2\left|\nabla\left(\frac{A}{nH}\right)\right|^2
			&=\sum_{i,j,k}
			\left|\nabla_kA_{ij}-\frac{A_{ij}}H\nabla_kH\right|^2\\
			&=|\nabla A|^2-2n^2|\nabla H|^2
			+\frac{|A|^2}{H^2}|\nabla H|^2\\
			&=|\nabla A|^2-n^2|\nabla H|^2.
		\end{align*}
		This proves~\eqref{h:eq:energydefect}.
		
		Next, Codazzi gives $\nabla_kA_{ij}=\nabla_jA_{ik}$ and hence
		\[
		\frac{A_{ij}}{nH}\,n\nabla_kH
		-\frac{A_{ik}}{nH}\,n\nabla_jH
		=nH\left[
		\nabla_j\left(\frac{A_{ik}}{nH}\right)
		-\nabla_k\left(\frac{A_{ij}}{nH}\right)\right].
		\]
		Squaring and summing over $i,j,k$, and using
		$|A/(nH)|^2=1$, yields
		\begin{equation}\label{h:eq:codazzi}
			2\left(n^2|\nabla H|^2-
			\left|\frac{A(\nabla H)}H\right|^2\right)
			\leq4\bigl(|\nabla A|^2-n^2|\nabla H|^2\bigr).
		\end{equation}
		Here $A(\nabla H)$ denotes the action of the shape operator on
		$\nabla H$. To make the spectral estimate explicit, put
		$b_i=\lambda_i/(nH)$ in a principal frame and
		$\mu=\max_i b_i^2$. Since $\sum_i b_i^2=1$, we have
		$0\leq\mu\leq1$ and $\sum_i b_i^4\geq\mu^2$. Therefore
		\begin{align*}
			\left|\frac{\Rm}{n^2H^2}\right|^2
			&=2\left(1-\sum_i b_i^4\right)\\
			&\leq2(1-\mu^2)
			=2(1-\mu)(1+\mu)\leq4(1-\mu).
		\end{align*}
		The same spectral bound gives
		\[
		n^2|\nabla H|^2-\left|\frac{A(\nabla H)}H\right|^2
		\geq n^2(1-\mu)|\nabla H|^2.
		\]
		Combining these inequalities with~\eqref{h:eq:codazzi} yields
		\begin{align*}
			n^2|\nabla H|^2\left|\frac{\Rm}{n^2H^2}\right|^2
			&\leq4n^2(1-\mu)|\nabla H|^2\\
			&\leq8\bigl(|\nabla A|^2-n^2|\nabla H|^2\bigr),
		\end{align*}
		which is~\eqref{h:eq:codazziestimates}. No division by $1-\mu$
		is used; the argument includes the rank-one case $\mu=1$.
		
		Finally, the components of $\Rm/(n^2H^2)$ are
		\[
		\frac{A_{ik}}{nH}\frac{A_{jl}}{nH}
		-\frac{A_{il}}{nH}\frac{A_{jk}}{nH}.
		\]
		Differentiation produces four products. Since $|A/(nH)|=1$, their
		tensor norms give~\eqref{h:eq:normalizedgauss}.
		
		At a point with $H=0$, we also have $A=0$. Evaluate the Laplacian of
		$|A|^2=n^2H^2$ at that point:
		\[
		2|\nabla A|^2=2n^2|\nabla H|^2.
		\]
		Thus the claimed difference is zero on the zero set.
	\end{pf}
	
	\subsection{The weighted curvature at its zero set}
	
	Whenever $A=0$, we define the expression $|A|^{(n-4)/2}\Rm$ to be
	zero. This convention is used throughout the rest of the proof. It is
	necessary only to interpret the apparently negative power when $n=3$;
	the following lemma proves that the extended tensor is regular.
	
	\begin{lemma}[Regularity of the weighted curvature]
		\label{h:prop:weightedcurvature}
		The tensor $|A|^{(n-4)/2}\Rm$, extended as above, is of class $C^1$.
		It satisfies
		\begin{align}
			\bigl||A|^{(n-4)/2}\Rm\bigr|^2
			&\leq2|A|^n,\label{h:eq:weightedtensor}\\
			\left|\nabla\bigl(|A|^{(n-4)/2}\Rm\bigr)\right|^2
			&\leq4(n^2+8)|A|^{n-2}
			\bigl(|\nabla A|^2-n^2|\nabla H|^2\bigr).
			\label{h:eq:weightedgradient}
		\end{align}
		Its norm $|A|^{(n-4)/2}|\Rm|$ is locally Lipschitz and satisfies
		the weak Kato inequality.
	\end{lemma}
	
	\begin{pf}
		The first estimate follows immediately from~\eqref{h:eq:gaussnorm}.
		Where $H\neq0$, scalar flatness gives
		\[
		|A|^{(n-4)/2}\Rm
		=|A|^{n/2}\frac{\Rm}{n^2H^2},
		\qquad |\nabla|A||^2=n^2|\nabla H|^2.
		\]
		The product rule and $|X+Y|^2\leq2|X|^2+2|Y|^2$ give
		\begin{align*}
			&\left|\nabla\bigl(|A|^{(n-4)/2}\Rm\bigr)\right|^2\\
			&\quad\leq
			\frac{n^2}{2}|A|^{n-2}n^2|\nabla H|^2
			\left|\frac{\Rm}{n^2H^2}\right|^2
			+2|A|^n\left|\nabla\left(\frac{\Rm}{n^2H^2}\right)\right|^2\\
			&\quad\leq
			4n^2|A|^{n-2}\bigl(|\nabla A|^2-n^2|\nabla H|^2\bigr)
			+32|A|^n\left|\nabla\left(\frac{A}{nH}\right)\right|^2\\
			&\quad=
			4(n^2+8)|A|^{n-2}
			\bigl(|\nabla A|^2-n^2|\nabla H|^2\bigr),
		\end{align*}
		where the last two lines use all three estimates of
		Lemma~\ref{h:lem:codazzi}.
		
		We check the zero extension independently. Differentiating the Gauss
		equation gives $|\nabla\Rm|\leq4|A||\nabla A|$. Together with
		\eqref{h:eq:gaussnorm} and $|\nabla|A||\leq|\nabla A|$, this gives,
		away from $A=0$,
		\begin{equation}\label{h:eq:zeroregularity}
			\left|\nabla\bigl(|A|^{(n-4)/2}\Rm\bigr)\right|
			\leq\left(4+\frac{|n-4|}{\sqrt2}\right)
			|A|^{(n-2)/2}|\nabla A|.
		\end{equation}
		Fix a point $x_0$ where $A(x_0)=0$. On a sufficiently small coordinate
		neighborhood, smoothness of $A$ implies
		$|A(x)|\leq C\,d_g(x,x_0)$ and $|\nabla A(x)|\leq C$.
		The tensor and its derivative therefore satisfy
		\[
		\bigl||A|^{(n-4)/2}\Rm\bigr|(x)
		\leq C\,d_g(x,x_0)^{n/2},
		\]
		and, at points where $A(x)\neq0$,
		\[
		\left|\nabla\bigl(|A|^{(n-4)/2}\Rm\bigr)\right|(x)
		\leq C\,d_g(x,x_0)^{(n-2)/2}.
		\]
		Fix a smooth tensor trivialization on this neighborhood. Its
		component norms are uniformly comparable to the metric norm, and its
		ordinary derivatives differ from covariant derivatives by bounded
		smooth connection coefficients times the tensor. Since $n/2>1$,
		the first estimate proves differentiability of every extended component
		at $x_0$, with derivative zero. Since $(n-2)/2>0$, the second
		estimate, together with the vanishing of the tensor, proves continuity
		of those ordinary derivatives across the zero set. For $n=3$ the two exponents are $3/2$ and $1/2$,
		respectively, so the same argument applies. Thus the extension is $C^1$,
		and~\eqref{h:eq:weightedgradient} holds at the zero set as well.
		The norm of a $C^1$ tensor is locally Lipschitz; its weak gradient is
		bounded almost everywhere by the norm of the tensor's covariant
		derivative. This is the required Kato inequality.
	\end{pf}
	
	\section{Weighted estimates and intrinsic flatness}
	\label{sec:gap}
	
	We now derive an energy inequality for the weighted curvature directly
	in terms of $H$, $A$, and $\Rm$. The weight $|A|^{n-2}$ in the
	integration by parts has two roles: it gives a nonnegative gradient
	term, and it makes the cutoff error integrable under the hypothesis
	$\int_M|A|^n<\infty$.
	
	\subsection{The scalar-flat Simons identity}
	
	The Euclidean hypersurface identity arising from the Simons
	calculation~\cite{Simons}, with normalized mean curvature, is
	\begin{equation}\label{h:eq:simons}
		\Delta A_{ij}
		=n\nabla_i\nabla_jH+nH(A^2)_{ij}-|A|^2A_{ij}.
	\end{equation}
	For this nonminimal form, one applies Codazzi to
	$\Delta A_{ij}=\sum_k\nabla_k\nabla_j A_{ik}$ and commutes
	covariant derivatives using the Gauss equation. The trace term becomes
	$n\nabla_i\nabla_jH$, and the curvature terms give
	$nH(A^2)_{ij}-|A|^2A_{ij}$; see also the calculation in
	\cite[Lemma~2.2]{LiXuZhou}. Contracting with $A$ and using
	$|A|^2=n^2H^2$ gives
	\[
	\frac12\Delta|A|^2
	=|\nabla A|^2+n\langle A,\nabla^2H\rangle
	+nH\tr(A^3)-n^4H^4.
	\]
	On the other hand,
	\[
	\frac12\Delta(n^2H^2)=n^2H\Delta H+n^2|\nabla H|^2.
	\]
	Subtracting these equal expressions and rearranging gives
	\[
	n\langle nHg-A,\nabla^2H\rangle
	=|\nabla A|^2-n^2|\nabla H|^2
	+nH\tr(A^3)-n^4H^4.
	\]
	Codazzi implies $\Div A=n\nabla H$, so $\Div(nHg-A)=0$. Thus
	\begin{equation}\label{h:eq:divergence}
		\Div\bigl(n(nHg-A)\nabla H\bigr)
		=|\nabla A|^2-n^2|\nabla H|^2
		+nH\tr(A^3)-n^4H^4.
	\end{equation}
	
	\subsection{Integration by parts and the cutoff estimate}
	
	\begin{lemma}
		\label{h:lem:regularized}
		Let $n\geq3$ and let $\phi$ be a real-valued, compactly
		supported Lipschitz function on $M$. Then the following identity holds:
		\begin{align}
			&\int_M\phi^2|A|^{n-2}
			\bigl(|\nabla A|^2-n^2|\nabla H|^2\bigr)\notag\\
			&\quad+(n-2)n^2\int_M\phi^2|A|^{n-4}
			\bigl[nH(nHg-A)\bigr](\nabla H,\nabla H)\notag\\
			&=\int_M\phi^2|A|^{n-2}
			\bigl(n^4H^4-nH\tr(A^3)\bigr)\notag\\
			&\quad-2n\int_M\phi |A|^{n-2}
			(nHg-A)(\nabla H,\nabla\phi).
			\label{h:eq:weightedidentity}
		\end{align}
		The tensor $|A|^{n-4}nH(nHg-A)$ is extended by zero at $A=0$;
		\eqref{h:eq:newtonbounds} bounds it by a constant times
		$|A|^{n-2}$, so that this extension is continuous.
		
	\end{lemma}
	
	\begin{pf}
		First take $\phi$ smooth and fix a compact set containing its support.
		Multiply~\eqref{h:eq:divergence} by
		$\phi^2(|A|^2+\delta^2)^{(n-2)/2}$, where $0<\delta\leq1$.
		This multiplier is smooth, so the divergence theorem gives an ordinary
		integration by parts. Since $\nabla|A|^2=2n^2H\nabla H$,
		\[
		\nabla\bigl((|A|^2+\delta^2)^{(n-2)/2}\bigr)
		=(n-2)n^2H(|A|^2+\delta^2)^{(n-4)/2}\nabla H.
		\]
		The resulting equality is exactly~\eqref{h:eq:weightedidentity}, with
		$|A|^{n-2}$ replaced by $(|A|^2+\delta^2)^{(n-2)/2}$ and
		$|A|^{n-4}$ by $(|A|^2+\delta^2)^{(n-4)/2}$.
		
		We give uniform domination for every type of term. When $n=3$,
		the regularized basic weight is bounded by $(1+|A|^2)^{1/2}$.
		The only negative power occurs in the quadratic term, and
		\eqref{h:eq:newtonbounds} gives
		\[
		\left|\frac{nH(nHg-A)}{\sqrt{|A|^2+\delta^2}}\right|
		\leq\frac{2(n-1)}{n}
		\frac{|A|^2}{\sqrt{|A|^2+\delta^2}}
		\leq\frac{2(n-1)}{n}|A|.
		\]
		Here the norm is the operator norm; equivalently the same bound holds
		as an inequality of quadratic forms. Thus the absolute values of the
		four types of integrands are bounded, up to fixed factors, by
		\begin{align*}
			&\phi^2(1+|A|^2)^{1/2}
			(|\nabla A|^2+n^2|\nabla H|^2),\\
			&\phi^2|A|\,|\nabla H|^2,\\
			&\phi^2(1+|A|^2)^{1/2}
			(n^4H^4+n|H|\,|\tr(A^3)|),\\
			&|\phi|(1+|A|^2)^{1/2}
			(n|H|+|A|)|\nabla H|\,|\nabla\phi|.
		\end{align*}
		They are integrable on the fixed compact support and independent of
		$\delta$. For $n\geq4$, both weight powers are nonnegative and
		uniformly bounded there for $0<\delta\leq1$, so the corresponding
		integrands have the same compact-support domination. At $A=0$ the
		quadratic tensor is zero for every $\delta$, while away from that
		set it converges to its stated zero extension. Dominated convergence
		therefore permits $\delta\downarrow0$ in all four terms and gives
		\eqref{h:eq:weightedidentity} for smooth $\phi$.
		
		Finally, a real-valued compactly supported Lipschitz $\phi$ can
		be approximated in local charts by real-valued smooth functions
		with supports in one fixed compact neighborhood, uniformly bounded
		in $L^\infty$ and
		converging in $W^{1,2}$ and pointwise almost everywhere. All tensor
		coefficients in~\eqref{h:eq:weightedidentity}, with the zero
		extensions already specified, are bounded on that neighborhood.
		The quadratic terms converge in $L^1$, and the cutoff term converges
		by the $L^2$ convergence of $\phi$ and its gradient. This proves the
		identity for all real-valued compactly supported Lipschitz functions.
		The approximants need not be nonnegative.
	\end{pf}
	
	\begin{proposition}[Weighted curvature energy]\label{h:prop:energy}
		For every nonnegative compactly supported Lipschitz function $\phi$,
		\begin{align}
			&\int_M
			\left|\nabla\bigl(\phi |A|^{(n-4)/2}|\Rm|\bigr)\right|^2
			\notag\\
			&\quad\leq
			2n(n^2+8)\int_M\phi^2|A|^{n-2}|\Rm|^2
			+C(n)\int_M|A|^n|\nabla\phi|^2,
			\label{h:eq:productenergy}
		\end{align}
		where one may take
		\[
		C(n)=\frac{32(n^2+8)(n-1)}{n(n-2)}+4.
		\]
		All weighted expressions use the zero extension in
		Lemma~\ref{h:prop:weightedcurvature}.
	\end{proposition}
	
	\begin{pf}
		Apply Lemma~\ref{h:lem:regularized} to obtain
		\eqref{h:eq:weightedidentity}. The second integral on the left is nonnegative. We estimate the final
		integral on the right in terms of it. On $H\neq0$, the
		Cauchy--Schwarz inequality for the nonnegative quadratic form
		$nH(nHg-A)$ gives
		\begin{align*}
			&2n\phi |A|^{n-2}
			|(nHg-A)(\nabla H,\nabla\phi)|\\
			&\quad\leq
			2n\phi |A|^{n-3}
			\sqrt{\bigl[nH(nHg-A)\bigr](\nabla H,\nabla H)}
			\sqrt{\bigl[nH(nHg-A)\bigr](\nabla\phi,\nabla\phi)}.
		\end{align*}
		Young's inequality now yields
		\begin{align*}
			&2n\phi |A|^{n-2}
			|(nHg-A)(\nabla H,\nabla\phi)|\\
			&\quad\leq
			\frac{(n-2)n^2}{2}\phi^2|A|^{n-4}
			\bigl[nH(nHg-A)\bigr](\nabla H,\nabla H)\\
			&\qquad+\frac2{n-2}|A|^{n-2}
			\bigl[nH(nHg-A)\bigr](\nabla\phi,\nabla\phi)\\
			&\quad\leq
			\frac{(n-2)n^2}{2}\phi^2|A|^{n-4}
			\bigl[nH(nHg-A)\bigr](\nabla H,\nabla H)
			+\frac{4(n-1)}{n(n-2)}|A|^n|\nabla\phi|^2.
		\end{align*}
		The final line uses~\eqref{h:eq:newtonbounds}. The same inequality
		holds at $A=0$ by continuity. Substitution
		into~\eqref{h:eq:weightedidentity} absorbs half of the nonnegative
		quadratic term. Dropping the remaining half gives
		\begin{align}
			&\int_M\phi^2|A|^{n-2}
			\bigl(|\nabla A|^2-n^2|\nabla H|^2\bigr)\notag\\
			&\quad\leq
			\int_M\phi^2|A|^{n-2}\bigl(n^4H^4-nH\tr(A^3)\bigr)
			+\frac{4(n-1)}{n(n-2)}\int_M|A|^n|\nabla\phi|^2.
			\label{h:eq:defectenergy}
		\end{align}
		
		Use~\eqref{h:eq:weightedgradient} on the left
		and~\eqref{h:eq:algebra} on the right to obtain
		\begin{align}
			&\int_M\phi^2
			\left|\nabla\bigl(|A|^{(n-4)/2}\Rm\bigr)\right|^2
			\notag\\
			&\quad\leq
			n(n^2+8)\int_M\phi^2|A|^{n-2}|\Rm|^2
			+\frac{16(n^2+8)(n-1)}{n(n-2)}
			\int_M|A|^n|\nabla\phi|^2.
			\label{h:eq:tensorenergy}
		\end{align}
		For clarity, the factor $n(n^2+8)$ is the product of
		$4(n^2+8)$ from~\eqref{h:eq:weightedgradient} and $n/4$ from the
		bound $n^4H^4-nH\tr(A^3)\leq(n/4)|\Rm|^2$.
		
		Finally, the product rule, the weak Kato inequality, and
		\eqref{h:eq:weightedtensor} give
		\begin{align*}
			&\int_M\left|\nabla\bigl(\phi |A|^{(n-4)/2}|\Rm|\bigr)\right|^2\\
			&\quad\leq
			2\int_M\phi^2\left|\nabla\bigl(|A|^{(n-4)/2}\Rm\bigr)\right|^2
			+2\int_M|A|^{n-4}|\Rm|^2|\nabla\phi|^2\\
			&\quad\leq
			2\int_M\phi^2\left|\nabla\bigl(|A|^{(n-4)/2}\Rm\bigr)\right|^2
			+4\int_M|A|^n|\nabla\phi|^2.
		\end{align*}
		Combining this with~\eqref{h:eq:tensorenergy}
		proves~\eqref{h:eq:productenergy}, with the stated $C(n)$.
	\end{pf}
	
	\subsection{Smallness, absorption, and intrinsic exhaustion}
	
	\begin{pf}[Proof of the intrinsic flatness assertion in
		Theorem~\ref{h:thm:main}]
		Assume $n\geq3$. The case $n=2$ and the classification of the
		immersion will be completed in Section~\ref{sec:classification}.
		
		From~\eqref{h:eq:beta},
		$\norm{H}_{L^n}<\beta_n$ implies
		$\norm{H}_{L^n}<1/(2nC_n)$. Thus~\eqref{h:eq:sobolev} is available
		with the bound~\eqref{h:eq:sobolevbound}. For a nonnegative cutoff $\phi$, Lemma~\ref{h:prop:weightedcurvature}
		shows that
		\[
		\phi |A|^{(n-4)/2}|\Rm|
		\]
		is an admissible test function. Applying
		\eqref{h:eq:sobolev} and Proposition~\ref{h:prop:energy} gives
		\begin{align}
			&\norm{\phi |A|^{(n-4)/2}|\Rm|}_{L^{2n/(n-2)}}^2
			\notag\\
			&\quad\leq
			2n(n^2+8)C_S\int_M\phi^2|A|^{n-2}|\Rm|^2
			+C(n)C_S\int_M|A|^n|\nabla\phi|^2.
			\label{h:eq:absorptionstart}
		\end{align}
		The first integral has exactly the critical factorization
		\[
		\phi^2|A|^{n-2}|\Rm|^2
		=|A|^2\bigl(\phi |A|^{(n-4)/2}|\Rm|\bigr)^2.
		\]
		H\"older's inequality with conjugate exponents $n/2$ and $n/(n-2)$
		therefore gives
		\begin{align*}
			\int_M\phi^2|A|^{n-2}|\Rm|^2
			&\leq
			\left(\int_M|A|^n\right)^{2/n}\\
			&\quad\cdot
			\left(\int_M
			\bigl(\phi |A|^{(n-4)/2}|\Rm|\bigr)^{2n/(n-2)}
			\right)^{(n-2)/n}.
		\end{align*}
		It follows that
		\begin{align}
			&\norm{\phi |A|^{(n-4)/2}|\Rm|}_{L^{2n/(n-2)}}^2
			\notag\\
			&\quad\leq
			2n(n^2+8)C_S\norm{A}_{L^n}^2
			\norm{\phi |A|^{(n-4)/2}|\Rm|}_{L^{2n/(n-2)}}^2
			\notag\\
			&\qquad+C(n)C_S\int_M|A|^n|\nabla\phi|^2.
			\label{h:eq:absorption}
		\end{align}
		
		We now check the smallness needed to absorb the curvature term.
		The equality $\norm{A}_{L^n}=n\norm{H}_{L^n}$
		and~\eqref{h:eq:sobolevbound} yield
		\begin{align}
			2n(n^2+8)C_S\norm{A}_{L^n}^2
			&\leq
			\frac{32n^3(n^2+8)(n-1)^2C_n^2}{(n-2)^2}
			\norm{H}_{L^n}^2\notag\\
			&<\frac14,
			\label{h:eq:smallcoefficient}
		\end{align}
		Indeed, substituting~\eqref{h:eq:beta} gives the exact identity
		\[
		\frac{32n^3(n^2+8)(n-1)^2C_n^2}{(n-2)^2}\,\beta_n^2
		=\frac14.
		\]
		Move this multiple of the left-hand side
		of~\eqref{h:eq:absorption} to the left. Since its remaining
		coefficient exceeds $3/4$, we obtain, in particular,
		\begin{equation}\label{h:eq:finalcutoff}
			\norm{\phi |A|^{(n-4)/2}|\Rm|}_{L^{2n/(n-2)}}^2
			\leq2C(n)C_S\int_M|A|^n|\nabla\phi|^2.
		\end{equation}
		Only the cutoff error remains.
		
		Fix $o\in M$. Completeness and Hopf--Rinow make closed intrinsic balls
		compact. For $r>0$, take the cutoff equal to one on
		$B_r(o)$, equal to $2-d_g(o,\cdot)/r$ on
		$B_{2r}(o)\setminus B_r(o)$, and zero outside $B_{2r}(o)$.
		Denote it by $\phi_r$. It has $0\leq\phi_r\leq1$ and
		$|\nabla\phi_r|\leq1/r$ almost everywhere. Moreover,
		\[
		\int_M|A|^n=n^n\int_M|H|^n<\infty.
		\]
		Using $\phi=\phi_r$ in~\eqref{h:eq:finalcutoff} gives
		\begin{align}
			&\left(\int_{B_r(o)}
			\bigl(|A|^{(n-4)/2}|\Rm|\bigr)^{2n/(n-2)}
			\right)^{(n-2)/n}\notag\\
			&\quad\leq
			\frac{2C(n)C_S}{r^2}\int_M|A|^n
			\longrightarrow0\qquad(r\to\infty).
			\label{h:eq:vanishing}
		\end{align}
		For any fixed $\rho>0$, the same integral over $B_\rho(o)$ is
		bounded by the left-hand side for all $r\geq\rho$. It must therefore
		be zero. Exhaustion by balls and continuity of the weighted curvature
		show that
		\[
		|A|^{(n-4)/2}\Rm\equiv0.
		\]
		At a point with $A\neq0$, the weight is a positive finite number, so
		$\Rm=0$. At $A=0$, the Gauss equation already gives $\Rm=0$.
		This proves intrinsic flatness on $M$. It remains to pass from intrinsic flatness to the affine-hyperplane classification.
	\end{pf}
	
	\begin{remark}
		The argument uses only the geometric quantities $H$, $A$, and $\Rm$,
		together with the cutoff and the Sobolev constant. For $n=4$, the
		weighted tensor is simply $\Rm$. For $n=3$, the Gauss bound
		$|\Rm|\leq\sqrt2|A|^2$ makes the weight harmless at $A=0$.
		No sign for $H$, positive lower bound for the Newton tensor,
		properness, or extrinsic volume-growth assumption is used.
	\end{remark}
	
	\begin{remark}[A larger sufficient threshold]
		Keeping the exact Sobolev constant in~\eqref{h:eq:sobolev}
		gives a larger threshold within the same argument. Indeed, if
		$h=\norm{H}_{L^n}$, then the coefficient in~\eqref{h:eq:absorption}
		is
		\[
		\left(
		\frac{2n(n-1)C_n\sqrt{2n(n^2+8)}\,h}
		{(n-2)(1-nC_nh)}
		\right)^2.
		\]
		It is less than one provided
		\begin{equation}\label{h:eq:improvedbeta}
			h<\widehat\beta_n:=
			\frac{n-2}
			{nC_n\bigl[(n-2)+2(n-1)\sqrt{2n(n^2+8)}\bigr]}.
		\end{equation}
		This condition also implies $nC_nh<1$. For a fixed
		$h<\widehat\beta_n$, absorption gives~\eqref{h:eq:finalcutoff}
		with a finite constant depending on $n$, $C_n$, and $h$ in
		place of $2C(n)C_S$, so the same exhaustion proves flatness.
		The main theorem uses the smaller $\beta_n$ to keep a uniform
		absorption margin. Neither threshold is claimed to be sharp.
	\end{remark}
	
	\section{Global classification of flat hypersurfaces}
	\label{sec:classification}
	
	Intrinsic flatness allows a complete hypersurface to bend along one
	direction. The finite mean curvature integral excludes this bending.
	We give the global argument separately, since integrability on $M$
	must not be confused with integrability on an infinite covering.
	
	\begin{proposition}[Classification of flat hypersurfaces]
		\label{h:prop:flatclassification}
		Let $n\geq2$ and let $F:(M^n,g)\to\R^{n+1}$ be a connected,
		complete, intrinsically flat immersion. If
		\[
		\int_M|H|^q\dd v_g<\infty
		\]
		for some $q>0$, then $A\equiv0$, $M$ is isometric to $\R^n$, and
		$F$ is an isometric embedding onto an affine hyperplane.
	\end{proposition}
	
	\begin{pf}
		Let $\pi:\widetilde M\to M$ be the Riemannian universal covering.
		The lifted metric is complete: a geodesic upstairs projects to a
		geodesic on the complete manifold $M$, and its continuation downstairs
		lifts for every finite parameter interval. It is also flat.
		A complete simply connected flat Riemannian $n$-manifold is globally
		isometric to $\R^n$. We therefore identify $\widetilde M$ with $\R^n$.
		
		We use the Hartman--Nirenberg theorem~\cite{HartmanNirenberg} in the following
		form: \emph{if $G:\R^n\to\R^{n+1}$, $n\geq2$, is an isometric immersion of the
			flat Euclidean space, then, after rigid motions of the domain and of the
			target, $G(s,y)=(\gamma(s),y)$ with $\gamma$ a unit-speed plane curve.}
		Here ``flat'' means that all sectional curvatures of the induced metric vanish.
		For $n\geq3$ the vanishing of the Gauss--Kronecker curvature alone would not
		suffice: the product of a nonflat surface of $\R^3$ with $\R$ is complete and
		has zero Gauss--Kronecker curvature, but it is not a cylinder over a plane
		curve.
		
		Apply the Hartman--Nirenberg cylinder
		theorem~\cite{HartmanNirenberg} to the complete isometric immersion
		$F\circ\pi:\R^n\to\R^{n+1}$. After Euclidean rigid motions in the
		domain and ambient space, it has the form
		\begin{equation}\label{h:eq:cylinder}
			(F\circ\pi)(s,y)=(\gamma(s),y),
			\qquad (s,y)\in\R\times\R^{n-1},
		\end{equation}
		where $\gamma:\R\to\R^2$ is a unit-speed smooth plane curve. The
		induced metric upstairs is $\dd s^2+|\dd y|^2$. If $\kappa(s)$ is
		the signed curvature of $\gamma$, the principal curvatures of this
		cylinder are $\kappa(s),0,\ldots,0$. Consequently,
		\begin{equation}\label{h:eq:cylindermean}
			|H|\circ\pi=\frac{|\kappa(s)|}{n}.
		\end{equation}
		
		Suppose that $\kappa(s_0)\neq0$. By continuity, there is a bounded
		open interval $I$ containing $s_0$ and a number $c>0$ such that
		$|\kappa(s)|\geq c$ for $s\in I$. Shrinking $I$ if necessary,
		we may also assume that $\gamma$ is injective on $I$. Indeed,
		$\gamma'(s_0)$ is a unit vector, and projection onto that direction
		has positive derivative in a sufficiently small neighborhood of $s_0$.
		
		The covering map $\pi$ is injective on the full strip
		$I\times\R^{n-1}$. To see this, suppose
		$\pi(s,y)=\pi(s',y')$ with $s,s'\in I$. Applying $F$ and using
		\eqref{h:eq:cylinder} gives
		\[
		(\gamma(s),y)=(\gamma(s'),y').
		\]
		Therefore $y=y'$ and, by injectivity of $\gamma|_I$, $s=s'$.
		Since $\pi$ is a local isometry and an open map, it identifies this
		strip isometrically with an open subset of $M$. We may thus integrate
		on that subset without any multiplicity.
		
		For each $L>0$, positivity of the integrand and
		\eqref{h:eq:cylindermean} give
		\begin{align}
			\int_M|H|^q\dd v_g
			&\geq\int_{I\times(-L,L)^{n-1}}
			\left(\frac{|\kappa(s)|}{n}\right)^q\dd s\dd y
			\notag\\
			&\geq\left(\frac cn\right)^q|I|(2L)^{n-1}.
			\label{h:eq:cylinderdivergence}
		\end{align}
		Here $|I|>0$ is the length of the interval. Since $n\geq2$, the
		right-hand side tends to infinity as $L\to\infty$, contradicting
		the assumed finite integral. Hence $\kappa$ vanishes identically.
		
		A unit-speed plane curve with zero curvature has constant tangent
		vector. Thus $\gamma(s)=\gamma(0)+s\gamma'(0)$, and
		$F\circ\pi$ is an injective affine isometric map onto an affine
		hyperplane. This also proves injectivity of $\pi$: if two points
		have the same projection, their images under $F\circ\pi$ agree,
		so the points coincide. As a bijective local isometry, $\pi$ is a
		global isometry. Consequently, $M$ itself is isometric to $\R^n$,
		$F$ has the claimed affine-hyperplane form, and $A\equiv0$.
	\end{pf}
	
	\begin{remark}
		Proposition~\ref{h:prop:flatclassification} requires finiteness of
		$\int|H|^q$ for just one exponent $q>0$, rather than smallness
		at the critical exponent. The critical exponent is needed in the
		preceding analytic step that establishes intrinsic flatness.
	\end{remark}
	
	\begin{pf}[Completion of the proof of Theorem~\ref{h:thm:main}]
		Section~\ref{sec:gap} proves $\Rm=0$ for $n\geq3$. In this
		case the hypothesis gives $\int_M|H|^n\dd v_g<\infty$, and
		Proposition~\ref{h:prop:flatclassification} applies with $q=n$.
		For $n=2$, the identity $R=2K=0$ gives intrinsic flatness
		directly, and the assumed finite integral with some $q>0$
		again permits application of Proposition~\ref{h:prop:flatclassification}.
		In either case, $A=0$, $M$ is isometric to $\R^n$, and $F$
		is an isometric embedding onto an affine hyperplane.
		
		Choose the isometry $\Psi:\R^n\to M$ given by this classification and
		an ambient rigid motion that sends the affine image to
		$\R^n\times\{0\}$. Absorbing the remaining Euclidean isometry into $\Psi$ gives
		the precise form~\eqref{h:eq:hyperplaneclassification}. This representation
		proves simple connectedness, injectivity, and total geodesicity. It also
		proves properness: the intersection of a compact ambient set with the
		closed affine hyperplane is compact, and its inverse image under the
		isometry is compact. Conversely, an affine hyperplane is complete and
		satisfies $A=H=R=0$, so it satisfies every stated smallness requirement.
		For $n=2$, only finiteness of $\int_M|H|^q\dd v_g$ for
		one exponent $q>0$ is needed, as asserted.
	\end{pf}
	
	\section{Testing the scalar-flat and smallness assumptions}
	\label{sec:scalar-necessity}
	The following examples test the scalar-flat and smallness hypotheses
	of Theorem~\ref{h:thm:main}. The rotational model in
	Example~\ref{prop:finite-test} is classical and belongs to the
	two-ended scalar-flat family studied by Hounie--Leite~\cite{HounieLeite}.
	Its finite total mean curvature norm is also noted in
	\cite[Introduction]{LiXuZhou}. We include a self-contained verification
	and compute its critical energy in the normalization used here.
	The singular normal graphs over cones constructed by
	de Lira--Soret~\cite{LiraSoret} provide further scalar-flat examples
	in a different setting.
	
	\begin{table}[htbp]
		\centering\small
		\begin{tabular}{@{}p{0.28\textwidth}p{0.25\textwidth}p{0.37\textwidth}@{}}
			\toprule
			Changed assumption & Family & Curvature and energy \\
			\midrule
			Remove $R=0$ & Example~\ref{prop:minimal-test}
			& $H=0$, but $R<0$ \\[0.4em]
			Replace $R=0$ by small $\norm{R}_{L^{n/2}}$
			& Example~\ref{prop:approximate-test}
			& $\norm{H}_{L^n}=O(\varepsilon)$,
			$\norm{R}_{L^{n/2}}=O(\varepsilon^2)$ \\[0.4em]
			Replace smallness by finiteness
			& Example~\ref{prop:finite-test}
			& $R=0$, $\int|H|^n=E_n>0$ \\[0.4em]
		\end{tabular}
		\caption{Nonflat examples testing different hypotheses of the gap theorem.}
		\label{tab:obstructions}
	\end{table}
	
	\subsection{Removing the scalar-flat condition}
	
	\begin{example}[Failure after removing $R=0$]\label{prop:minimal-test}
		For every $n\geq2$, there is a connected, complete, properly embedded
		nonflat hypersurface in $\R^{n+1}$ with $H\equiv0$.
	\end{example}
	\begin{pf}[Verification]
		Take the product of the classical catenoid with $\R^{n-2}$:
		\[
		F(\theta,t,z)
		=(\cosh t\cos\theta,\cosh t\sin\theta,t,z),\qquad
		(\theta,t,z)\in\Sph^1\times\R\times\R^{n-2}.
		\]
		The induced metric is
		\[
		g=\cosh^2t\,(\dd\theta^2+\dd t^2)+|\dd z|^2.
		\]
		The metric dominates the complete product metric. The catenoid is
		embedded, and the $t$ and $z$ coordinates show properness. Its
		principal curvatures are $\operatorname{sech}^2t$,
		$-\operatorname{sech}^2t$, and $n-2$ zeros. Thus $H=0$, whereas
		\[
		R=-2\operatorname{sech}^4t<0.
		\]
		In particular, the mean curvature norm vanishes and the induced
		metric is not flat.
	\end{pf}
	
	By \eqref{eq:contracted-gauss}, a scalar-flat minimal immersion is
	totally geodesic in every codimension.
	
	\subsection{Small scalar curvature versus exact scalar flatness}
	
	\begin{example}[Small scalar curvature does not give exact flatness]
		\label{prop:approximate-test}
		For every $n\geq2$, there are complete, properly embedded hypersurface
		graphs $G_\varepsilon:\R^n\to\R^{n+1}$, $0<\varepsilon\leq1$, such that
		\[
		\norm{H_\varepsilon}_{L^n}
		+\norm{\mathrm{II}_\varepsilon}_{L^n}\leq C\varepsilon,
		\qquad \norm{R_{g_\varepsilon}}_{L^{n/2}}\leq C\varepsilon^2,
		\]
		but $\Rm_{g_\varepsilon}(0)\neq0$ for every such $\varepsilon$.
	\end{example}
	\begin{pf}[Verification]
		Choose $\varphi\in C_c^\infty(\R^n)$ equal to $|x|^2/2$ near
		the origin, and put $G_\varepsilon(x)=(x,\varepsilon\varphi(x))$.
		Every smooth entire Euclidean graph is complete, and this graph is
		proper since $|G_\varepsilon(x)|\geq|x|$. Its metric is uniformly
		comparable to the Euclidean metric for $0<\varepsilon\leq1$.
		The second fundamental form is supported in a fixed compact set and
		has norm at most $C\varepsilon$. Also $|H|\leq
		|\mathrm{II}|/\sqrt n$, and
		\eqref{eq:contracted-gauss} gives $|R|\leq(n+1)|\mathrm{II}|^2$.
		The asserted integral bounds follow. At the origin the gradient of
		the height vanishes and its Hessian is $\varepsilon I$, so
		\[
		K(e_i\wedge e_j)|_0=\varepsilon^2\quad(i\neq j),\qquad
		R(0)=n(n-1)\varepsilon^2>0.
		\]
		Thus the metric is not flat.
	\end{pf}
	
	Thus small $\norm{R}_{L^{n/2}}$ cannot replace $R=0$ in an exact
	flatness theorem. A quantitative question is discussed in
	Section~\ref{sec:questions}.
	
	\subsection{Finite energy versus an energy gap}
	For the classical rotational scalar-flat family recalled below,
	we compute the total mean curvature energy explicitly. This
	calculation shows that finite critical energy does not imply flatness.
	
	\begin{example}[Nonflat scalar-flat hypersurfaces with finite energy]
		\label{prop:finite-test}
		For every $n\geq3$ and $a>0$, there is a complete, properly embedded,
		nonflat scalar-flat hypersurface diffeomorphic to
		$\R\times\Sph^{n-1}$, whose total mean curvature integral is
		\begin{equation}\label{eq:rotational-energy}
			E_n:=\int|H|^n\dd v_g
			=\frac{2^{1-n}|\Sph^{n-1}|}{n-2}
			\mathrm B\left(\frac n2,\frac12\right)>0.
		\end{equation}
		Here $\mathrm B$ denotes the Euler beta function. In particular,
		the integral is independent of $a$, and $E_4=\pi^2/6$.
	\end{example}
	\begin{pf}[Verification]
		Let $r$ solve
		\[
		r''=\frac{n-2}{2}a^{n-2}r^{-(n-1)},\qquad
		r(0)=a,\quad r'(0)=0.
		\]
		Multiplication by $2r'$ and integration give
		\[
		(r')^2=1-(a/r)^{n-2}.
		\]
		The solution is even, is increasing for $s>0$, stays above $a$,
		and has $|r'|\leq1$. It therefore extends smoothly to all
		$s\in\R$. Its strict convexity implies $r(s)\to\infty$ as
		$|s|\to\infty$. Put
		\[
		z(s)=\int_0^s (a/r(u))^{(n-2)/2}\dd u,
		\qquad F(s,\omega)=(r(s)\omega,z(s)).
		\]
		Since $(r')^2+(z')^2=1$, the induced metric is
		$\dd s^2+r(s)^2g_{\Sph^{n-1}}$. It dominates a complete product
		metric and is complete. The function $z$ is strictly increasing,
		so $F$ is injective; $r(s)\to\infty$ and compactness of
		$\Sph^{n-1}$ give properness.
		
		With a consistent choice of normal, the tangential principal
		curvature is
		\[
		\lambda=\frac{z'}r=a^{(n-2)/2}r^{-n/2}
		\]
		with multiplicity $n-1$, and the radial principal curvature is
		$-r''/z'=-(n-2)\lambda/2$. Thus
		\[
		R=(n-1)\bigl(-(n-2)\lambda^2+(n-2)\lambda^2\bigr)=0,
		\qquad |H|=\frac{\lambda}{2}.
		\]
		A plane tangent to the sphere has sectional curvature $\lambda^2>0$,
		so the hypersurface is nonflat. Integrating over its two halves yields
		\[
		\int|H|^n\dd v_g
		=2^{1-n}|\Sph^{n-1}|a^{n(n-2)/2}
		\int_a^\infty
		\frac{r^{n-1-n^2/2}}{\sqrt{1-(a/r)^{n-2}}}\dd r.
		\]
		The substitution $u=(a/r)^{n-2}$ gives
		\[
		a^{n(n-2)/2}r^{n-1-n^2/2}\dd r
		=-\frac{1}{n-2}u^{n/2-1}\dd u.
		\]
		Hence the energy equals
		\[
		\frac{2^{1-n}|\Sph^{n-1}|}{n-2}
		\int_0^1u^{n/2-1}(1-u)^{-1/2}\dd u,
		\]
		which proves~\eqref{eq:rotational-energy} and explicitly shows
		cancellation of the scale $a$. This beta integral is finite because
		its powers at $u=0$ and $u=1$ are $n/2-1>-1$ and $-1/2>-1$.
	\end{pf}
	
	The critical energy is unchanged by dilation of the neck radius $a$.
	\section{Further questions}\label{sec:questions}

	For $n\geq3$, let $b_n^*$ be the supremum of all $b>0$ such
	that every complete connected scalar-flat Euclidean hypersurface
	with $\norm{H}_{L^n}<b$ is intrinsically flat. Theorem~\ref{h:thm:main}
	and Example~\ref{prop:finite-test} give
	\begin{equation}\label{eq:optimal-threshold}
		\beta_n\leq b_n^*\leq E_n^{1/n}.
	\end{equation}
	The refinement~\eqref{h:eq:improvedbeta} improves the lower bound
	to $\widehat\beta_n\leq b_n^*$. Does the rotational family realize
	the least critical energy among nonflat complete scalar-flat
	hypersurfaces? The nonsharp estimates used here do not identify an
	equality model.
	
	\section*{Acknowledgements}
	The author thanks Professor Xingwang Xu for bringing him to this
	problem several years ago and for sharing his ideas.
	\section*{Disclosure on AI assistance}
	The principal mathematical ideas were formulated by the author.
	AI-assisted tools were used to help check selected computations and
	improve the language of the manuscript. The author verified and
	completed all mathematical arguments and takes full responsibility
	for the content.

	\medskip
	\noindent{\small School of Mathematics, Yangzhou University, Yangzhou, China\\
		E-mail: \href{mailto:zhoujiuru@yzu.edu.cn}{zhoujiuru@yzu.edu.cn}}
	
\end{document}